\documentclass{amsart}

\usepackage{amssymb,stmaryrd,mathrsfs,dsfont,amsmath}
\usepackage{colonequals} 
\usepackage{tikz}
\usepackage{float}

\theoremstyle{plain}
\newtheorem{theorem}{Theorem}[section]
\newtheorem{lemma}[theorem]{Lemma}
\newtheorem{proposition}[theorem]{Proposition}
\newtheorem{corollary}[theorem]{Corollary}

\theoremstyle{definition}

\newtheorem{example}[theorem]{Example}

\newtheorem{remark}[theorem]{Remark}

\DeclareMathOperator{\N}{\mathbb{N}}
\DeclareMathOperator{\Z}{\mathbb{Z}}

\input xy

\usepackage{microtype}
\begin{document}
	
	
	\title[On the Independence Number of the Prime-Coprime Graph]{On the Independence Number of the Prime-Coprime Graph of a Finite Group}
			
\author[Ravi Ranjan, Shubh N. Singh, Surbhi Kumari, Shidra Jamil]{\bfseries Ravi Ranjan \and Shubh Narayan Singh \\
Surbhi Kumari \and  Shidra Jamil}
	
	\address{Department of Mathematics, Central University of South Bihar, Gaya--824236, Bihar, India}
	\email{raviranjan23@cusb.ac.in}
	\email{shubh@cub.ac.in}
	\email{18sept.surbhi@gmail.com}
	\email{shidrajamil136@gmail.com}

	

\begin{abstract}
The prime-coprime graph $\Theta(G)$ of a finite group $G$ is the simple graph with vertex set $G$, where two distinct elements are adjacent whenever the greatest common divisor of their orders is either $1$ or a prime. We characterize all finite groups $G$ for which $\Theta(G)$ is a split graph. We establish a general lower bound for the independence number of $\Theta(G)$ of an arbitrary finite group $G$. Moreover, we explicitly compute the independence number of $\Theta(G)$ for several distinguished families of finite groups, including cyclic, dihedral, dicyclic, and semidihedral groups.
\end{abstract}
	
\maketitle

\section{Introduction}
Graphs arising from finite groups have been extensively studied over the past two decades; see the survey in \cite{camron-ijgt22}. These graphs highlight the deep interplay between group theory and graph theory, 
allowing group-theoretic properties to be examined through graph invariants. Among recent constructions, the concept of prime–coprime graphs, introduced in 2021, has attracted considerable attention; see \cite{adhikari-nsjom22, hao-2022, li-2022, shubh-arxiv25, sehgal-jmcs21, sehgal-25, sehgal-ijpam2024, sehgal-jmcsc-21}.

\vspace{0.4mm}
Let $G$ be a finite group. The \emph{prime-coprime graph} of $G$, denoted by $\Theta(G)$, is the simple graph with vertex set $G$, where two distinct vertices are adjacent whenever the greatest common divisor of their orders is either $1$ or a prime \cite{adhikari-nsjom22}. This graph is also referred as the \emph{coprime order graph} in the literature; see \cite{hao-2022, li-2022}. Adhikari and Banerjee \cite{adhikari-nsjom22} characterized the completeness and Eulerian properties of $\Theta(G)$, and determined its domination number, diameter, and girth. For specific groups $G$, they further studied Hamiltonicity, planarity, vertex connectivity, and the signless Laplacian spectrum of $\Theta(G)$. Hao et al. \cite{hao-2022} classified finite groups $G$ for which $\Theta(G)$ is planar, and determined its vertex connectivity for finite cyclic, dihedral, and dicyclic groups. Li et al. proved that for any fixed integer $k\geq 1$, there are only finitely many finite groups $G$ such that $\Theta(G)$ has (non)orientable genus $k$ \cite{li-2022}. In~\cite{shubh-arxiv25}, the first two authors studied the Hamiltonicity and clique number of $\Theta(G)$ for finite cyclic, dihedral, and dicyclic groups.

\vspace{0.4mm}

The determination of the independence number of a graph is generally NP-hard. This graph invariant has been investigated for certain graphs associated with finite groups; see \cite{camron-gc20, hamm-ijgt21, jitender-bmmss21, ma-im18, jitender-ca21}. In this paper, we focus on the independence number of prime–coprime graphs of finite groups. The subsequent sections of this paper are organized as follows. Section 2 provides preliminary definitions and notation. In Section 3, we characterize all finite groups $G$ for which $\Theta(G)$ is a split graph. Furthermore, we establish a general lower bound for the independence number of $\Theta(G)$ of an arbitrary finite group $G$. Moreover, we explicitly compute the independence number of $\Theta(G)$ for the following four distinguished families of finite groups:
\begin{align*}
\text{Cyclic groups }\,\, \Z_n &= \langle a \, \colon  a^n = 1\rangle, \quad n\geq 1;\\
\text{Dihedral groups }\,\, D_{2n} &= \langle a, b \, \colon  a^n = 1,\,  b^2 = 1,\, bab^{-1} = a^{-1}\rangle, \quad n\geq 3;\\
\text{Dicyclic groups }\,\,Q_{4n} &= \langle a, b \, \colon  a^{2n} = 1,\,  b^4 = 1,\, bab^{-1} = a^{-1}\rangle, \quad n\geq 3;\\
\text{Semidihedral groups }\,\, SD_{8n} &= \langle a, b \, \colon  a^{4n} = 1,\,  b^2 = 1,\, bab^{-1} = a^{2n-1}\rangle, \quad n\geq 3.
\end{align*}


\section{Preliminaries and Notation}


\vspace{0.4mm}
The cardinality of a set $X$ is denoted by $|X|$. For sets $A$ and $B$, their disjoint union is denoted by $A\sqcup B$, and the set difference is denoted by $A\setminus B = \{x\in A \colon x\notin B\}$.
We denote by $\N$ the set of all positive integers. For any integers $a, b\in \N$, we denote by $(a,b)$ the greatest common divisor (gcd, for short) of $a$ and $b$. The \emph{canonical prime factorization} of an integer $n>1$ is a unique representation $n= p_1^{k_1}\cdots p_t^{k_t}$, where $p_1<\cdots < p_t$ are primes and $k_1, \ldots, k_t \in \N$. The number of distinct prime divisors of an integer $n>1$ is denoted by $\omega(n)$. A \emph{semiprime} is a composite integer that is the product of two (possibly equal) primes. For every positive integer $n$, we denote by $\varphi(n)$ Euler’s totient function of $n$. Unless stated otherwise, all groups considered are finite and nontrivial. Let $G$ be a group. For every element $g\in G$, we denote by $|g|$ the order of $g$, and by $\langle g \rangle$ the cyclic subgroup generated by $g$. The order of $G$ is denoted by $|G|$, and $\pi(G)$ denotes the set of all prime divisors of $|G|$. We set $P(G) = \bigl\{g\in G\ \colon |g| \in \pi(G)\cup \{1\}\bigr\}$. The cyclic group of order $n$ is denoted by $\Z_n$.

\vspace{0.5mm}
Throughout the paper, all graphs are assumed to be simple. Let $\Gamma$ be a graph. We denote by $V(\Gamma)$ (resp.~ $E(\Gamma)$) the set of all vertices (resp.~ edges) of $\Gamma$. For any two distinct vertices $u,v \in V(\Gamma)$, we denote by $uv$ (or $vu$) the edge of $\Gamma$ with endpoints $u$ and $v$. A \emph{dominating vertex} of $\Gamma$ is a
vertex that is adjacent to every other vertex of $\Gamma$. For every positive integer $m$, we use the notation $m\Gamma$ to denote the disjoint union of $m$ copies of $\Gamma$. For any vertex-disjoint subgraphs $\Gamma_1$ and $\Gamma_2$ of $\Gamma$, we write $\Gamma_1 \sim \Gamma_2$ if every vertex of $\Gamma_1$ is adjacent to every vertex of $\Gamma_2$ in $\Gamma$. We denote by $K_n$ the complete graph on $n$ vertices, and by $E_n$ the empty graph on $n$ vertices. Let $S$ be a nonempty subset of $V(\Gamma)$. The subgraph of $\Gamma$ induced by $S$ is denoted by $\Gamma[S]$, and $\Gamma - S$ denotes the graph obtained by removing the vertices of $S$ from $\Gamma$. The set $S$ is called a \emph{clique} (resp.~\emph{independent set}) of $\Gamma$ if $\Gamma[S]$ is complete (resp.~ empty). An independent set of $\Gamma$ is said to be \emph{maximal} if it is inclusion-maximal, and \emph{maximum} if it has the maximum cardinality among all independent sets of $\Gamma$. The \emph{independence number} of $\Gamma$, denoted by $\alpha(\Gamma)$, is the cardinality of a maximum independent set of $\Gamma$. We write $\Gamma_1 \cong \Gamma_2$ whenever graphs $\Gamma_1$ and $\Gamma_2$ are isomorphic. A graph $\Gamma$ is called a \emph{split} graph if its vertex set $V(\Gamma)$ admits a partition into two disjoint nonempty subsets $K$ and $I$ such that $K$ forms a clique and $I$ constitutes an independent set. The pair $(C, I)$ is then referred to as a \emph{split partition} of $\Gamma$. Moreover, if every vertex of $K$ is adjacent to every vertex of $I$, then $\Gamma$ is said to be a \emph{complete split graph}. It was established in \cite{foldes-hammer-77} that a graph is split if and only if it contains no induced subgraph isomorphic to $C_4$, $C_5$, or $2K_2$. Let $H$ be a graph with vertex set $\{1,2,\dots,t\}$, and let 
$\mathcal{G}=\{\Gamma_1,\Gamma_2,\dots,\Gamma_t\}$ be a family of pairwise vertex‑disjoint graphs. The \emph{$H$‑join} of $\mathcal{G}$, denoted $H[\mathcal{G}]$, is the graph obtained by replacing each vertex $i\in V(H)$ with the graph $\Gamma_i$, and for each edge $ij\in E(H)$, adding all possible edges between $V(\Gamma_i)$ and $V(\Gamma_j)$. Note that the \emph{join} of vertex‑disjoint graphs $\Gamma_1$ and $\Gamma_2$, denoted by $\Gamma_1 \vee \Gamma_2$, is precisely the $H$‑join of $\{\Gamma_1,\Gamma_2\}$ when $H=K_2$.

\vspace{0.5mm}
For terminology and notation concerning graphs, groups, and number theory not introduced in this paper, we refer the reader to \cite{west-b01}, \cite{gallian-b21}, and \cite{burton-b07}, respectively. We finally recall the following results that will be used in the sequel.

\begin{theorem}(\cite[Theorem~3.5]{adhikari-nsjom22})\label{th_thetaG-complete-iff-G-no-compo-order}
Let $G$ be a finite group. Then $\Theta(G)$ is complete if and only if $G$ has no elements of composite order.
\end{theorem}

\begin{remark}\cite[Theorem~3.1(b)]{adhikari-nsjom22}\label{re_dominating-vertex}
Let $G$ be a group. Then every element of $P(G)$ is a dominating vertex of $\Theta(G)$.
\end{remark}

\begin{theorem}(\cite[Theorem~4.4]{gallian-b21})\label{th_phi-d-elems-in-Zn}
If $d$ is a positive divisor of $n$, then the number of elements of order $d$ in a cyclic group of order $n$ is $\varphi(d)$.
\end{theorem}

\begin{theorem}(\cite[Theorem~7.3]{burton-b07})\label{th_value-of-phi-n}
If an integer $n>1$ has the canonical prime factorization $n = p_1^{k_1} p_2^{k_2} \cdots p_r^{k_r}$, then 
\[\varphi(n) = \bigl(p_1^{k_1}-p_1^{k_1-1} \bigr) \bigl(p_2^{k_2}-p_2^{k_2-1}\bigr) \cdots \bigl(p_r^{k_r}-p_r^{k_r-1}\bigr).\] 
\end{theorem}

\section{Main Results}
In this section, we characterize all finite groups $G$ for which $\Theta(G)$ is a split graph. Furthermore, we establish a general lower bound for the independence number $\alpha\bigl(\Theta(G)\bigr)$ of an arbitrary finite group $G$. Moreover, we explicitly compute the independence number $\alpha\bigl(\Theta(G)\bigr)$ for the four distinguished families of finite groups, namely cyclic, dihedral, dicyclic, and semidihedral groups.

\vspace{0.5mm}

Our first main result is the following theorem, which characterizes all finite groups $G$ for which $\Theta(G)$ is a split graph.

\begin{theorem}\label{th_complete-split-grph-class}
Let $G$ be a group. Then $\Theta(G)$ is a split graph if and only if one of the following conditions holds:
	\begin{enumerate}
		\item[\rm(i)] $G\setminus P(G) = \varnothing$.
		\item[\rm(ii)] $G\setminus P(G)$ consists exclusively of elements of order $p^k$, where $p$ is a fixed prime and $k\geq 2$.
		\item[\rm(iii)]  $G\setminus P(G)$ consists exclusively of elements of order $pq$, where $p$ and $q$ are fixed distinct primes. 
		\end{enumerate}
\end{theorem}

\begin{proof}[\textbf{Proof}]
Assume that $\Theta(G)$ is a split graph. If $\Theta(G)$ is a complete graph, then $P(G) = G$ by Theorem~\ref{th_thetaG-complete-iff-G-no-compo-order}, and hence the condition \rm(i) holds. Suppose that $\Theta(G)$ is not complete. By Theorem~\ref{th_thetaG-complete-iff-G-no-compo-order}, it follows that $G\setminus P(G) \neq \varnothing$. Choose $g\in G\setminus P(G)$. Note that $|g|$ is composite. We claim $\omega(|g|)\leq 2$. Suppose to the contrary that $\omega(|g|)\geq 3$. Then there exist distinct primes $p_i, p_j, p_{\ell}$ dividing $|g|$. Therefore $G$ contains at least $\varphi(p_ip_j)$ elements of order $p_ip_j$, and at least $\varphi(p_jp_{\ell})$ elements of order $p_jp_{\ell}$. Observe that $\varphi(p_ip_j), \varphi(p_jp_{\ell})\geq 2$, and choose $u,v,x,y\in G$ with $|u| = |v| = p_ip_j$ and $|x| = |y| = p_jp_{\ell}$. 
Then the subset $\{u,v,x,y\}$ induces a subgraph of $\Theta(G)$ isomorphic to $2K_2$. This contradicts the assumption that $\Theta(G)$ is split. Hence $\omega(|g|)\leq 2$. Condition \rm(ii) holds immediately if $\omega(|g|)= 1$. Finally, suppose that $\omega(|g|)= 2$. We claim that $|g| = pq$ for distinct primes $p, q$. Suppose to the contrary that $|g|=p^{a}q^{b}$ with $a+b\geq 3$. Without loss of generality, assume that $a\geq 2$. Then $G$ contains at least $\phi(p^{a})$ elements of order $p^{a}$, and at least $\phi(pq)$ elements of order $pq$. Observe that $\phi(p^{a}), \phi(pq) \geq 2$, and choose $u,v,x,y\in G$ with $|u| = |v| = p^a$ and $|x| = |y| = pq$. Then the subset $\{u,v,x,y\}$ induces a subgraph of $\Theta(G)$ isomorphic to $2K_2$. This contradicts the assumption that $\Theta(G)$ is split. Hence $|g| = pq$ for distinct primes $p,q$, and thus the condition \rm(iii) holds. 

\vspace{0.4mm}
 
Conversely, assume that one of the three conditions holds. If $G\setminus P(G) = \varnothing$, then by Theorem~\ref{th_thetaG-complete-iff-G-no-compo-order}, the graph $\Theta(G)$ is complete. Therefore $\Theta(G)$ is split.

\vspace{0.4mm}
Suppose $G\setminus P(G) \neq\varnothing$. We consider the two distinct cases separately.

\vspace{0.5mm}
\noindent \textbf{Case 1:} Suppose that every element of $G\setminus P(G)$ has order $p^k$ for some fixed prime $p$ and $k\geq 2$. Then any two distinct $g,h \in G\setminus P(G)$ are non-adjacent in $\Theta(G)$, since $(|g|,|h|)$ is composite. Therefore $G \setminus P(G)$ is an independent set of $\Theta(G)$.

\vspace{0.5mm}
\noindent \textbf{Case 2:} Suppose that every element of $G\setminus P(G)$ has order $pq$ for fixed distinct primes $p,q$. Then any two distinct $g,h \in G\setminus P(G)$ are non-adjacent in $\Theta(G)$, since $(|g|,|h|)$ is composite.
Therefore $G \setminus P(G)$ is an independent set of $\Theta(G)$.

\vspace{0.5mm}
In either case, $G \setminus P(G)$ is an independent set of $\Theta(G)$. By Remark~\ref{re_dominating-vertex}, $P(G)$ is a clique of $\Theta(G)$. Hence $\Theta(G)$ is split.	
\end{proof}

\begin{corollary}\label{cor_split-implies-complete-split}
Let $G$ be a group. If $\Theta(G)$ is a split graph, then $\Theta(G)$ is a complete split graph.
\end{corollary}

\begin{proof}[\textbf{Proof}]
If $G\setminus P(G) = \varnothing$, 
then by Theorem~\ref{th_thetaG-complete-iff-G-no-compo-order}, $\Theta(G)$ is complete. Therefore $\Theta(G)$ is complete split.

\vspace{0.4mm}
Suppose $G\setminus P(G) \neq \varnothing$. Since $\Theta(G)$ is split, Theorem~\ref{th_complete-split-grph-class} provides the split partition $\bigl(P(G),\,  G\setminus P(G)\bigr)$ of $\Theta(G)$. By Remark~\ref{re_dominating-vertex}, $P(G)$ is a clique and every element of $P(G)$ 
is dominating in $\Theta(G)$. Hence $\Theta(G)$ is complete split.
\end{proof}


\begin{corollary}\label{cor_ind-no-of-split-graph}
Let $G$ be a group. If $\Theta(G)$ is a split graph, then
\begin{equation*}
\alpha\bigl(\Theta(G)\bigr) =
	\begin{cases}
		1 & \text{if $G\setminus P(G) = \varnothing$},\\
		|G|- |P(G)| & \text{if $G\setminus P(G)\neq \varnothing$}.
	\end{cases}
\end{equation*}
 \end{corollary}

\begin{proof}[\textbf{Proof}]
If $G\setminus P(G) = \varnothing$, then by Theorem~\ref{th_thetaG-complete-iff-G-no-compo-order}, the graph $\Theta(G)$ is complete. Therefore $\alpha\bigl(\Theta(G)\bigr) = 1$. 

\vspace{0.4mm}
Suppose $G\setminus P(G)\neq \varnothing$. Since $\Theta(G)$ is split, Theorem~\ref{th_complete-split-grph-class} provides the split partition $\bigl(P(G),\, G\setminus P(G)\bigr)$ of $\Theta(G)$. By Corollary~\ref{cor_split-implies-complete-split}, $\Theta(G)$ is complete split, whence $G \setminus P(G)$ forms the maximum independent set. Therefore $\alpha\bigl(\Theta(G)\bigr) = |G\setminus P(G)| = |G|- |P(G)|$.
\end{proof}


In the remainder of the paper, we assume that $n$ is a composite integer and that $G$ denotes a group of order $n$. For each $n$, we define 
\[SP(n) = \bigl\{d\in \N \ \colon d \text{ is a semiprime divisor of } n\bigr\}.\]

Furthermore, for each $d\in SP(n)$, we define \[I_d(G)=\bigl\{g\in G\ \colon d \text{ divides } |g|\bigr\}.\]


\begin{lemma}\label{le_Id-maximal-set}
Let $d\in SP(n)$ such that $I_d(G) \neq \varnothing$. Then $I_d(G)$ forms a maximal independent set of $\Theta(G)$.
\end{lemma}

\begin{proof}[\textbf{Proof}]
To show first that $I_d(G)$ is an independent set of $\Theta(G)$, let $g,h\in I_d(G)$ with $g\neq h$. Then, by definition of $I_d(G)$, the semiprime $d$ divides both $|g|$ and $|h|$. Therefore $d$ divides $(|g|, |h|)$, and so $(|g|, |h|)$ is composite. Consequently, no edge of $\Theta(G)$ joins $g$ and $h$, and hence $I_d(G)$ is an independent set of $\Theta(G)$.

\vspace{0.5mm}
Suppose to the contrary that $I_{d}(G)$ is not maximal. Then there exists $x\in G\setminus I_{d}(G)$ such that $I_{d}(G)\cup \{x\}$ is independent in $\Theta(G)$. By definition of $I_{d}(G)$, we have $d\nmid |x|$. Therefore, for any $g\in I_{d}(G)$ with $|g| = d$, we have $(|x|, |g|)\in \{1\}\cup \pi(G)$. Hence $x$ is adjacent to such $g\in I_{d}(G)$ in $\Theta(G)$.
This contradicts the assumption that $I_{d}(G)\cup \{x\}$ is independent. Hence $I_{d}(G)$ is maximal.
\end{proof}


Lemma~\ref{le_Id-maximal-set} immediately provides the following proposition, which establishes a general lower bound for the independence number $\alpha\bigl(\Theta(G)\bigr)$ for any group $G$.

\begin{proposition}\label{pr_lower-bound-for-ind-num}
Let $G$ be a group of order $n$. Then 
	\[\alpha\bigl(\Theta(G)\bigr) \geq \max\bigl\{|I_d(G)|\colon d\in SP(n)\bigr\}.\]
\end{proposition}


In order to render the lower bound in Proposition~\ref{pr_lower-bound-for-ind-num} explicit for cyclic groups case, it is first necessary to determine $|I_d(\mathbb{Z}_{n})|$ for each $d \in SP(n)$. The following two results establish this computation.

\begin{lemma}\label{le_card-of-I_d}
Let $d\in SP(n)$. Then  
\[|I_{d}(\Z_{n})| = \sum_{\substack{d\mid \ell\\ \ell\mid n}}\varphi(\ell).\] 
\end{lemma}

\begin{proof}[\textbf{Proof}]
Recall that $I_{d}(\Z_{n})=\bigl\{g\in \Z_{n}\colon d \text{ divides } |g|\bigr\}$. 
Let $g\in I_{d}(\Z_{n})$ with $|g| = \ell$. Then $d\mid \ell$ and $\ell\mid n$. By Theorem~\ref{th_phi-d-elems-in-Zn}, the group $\Z_{n}$ contains exactly $\varphi(\ell)$ elements of order $\ell$. Consequently, the elements of $I_{d}(\Z_{n})$ are exactly those of order $\ell$ with $d\mid \ell$ and $\ell\mid n$. Hence, summing $\varphi(\ell)$ over all such divisors $\ell$ of $n$ with $d\mid \ell$ provides
$|I_{d}(\Z_{n})| = \sum_{\substack{d\mid \ell\\ \ell\mid n}}\varphi(\ell)$.
\end{proof}


\begin{proposition}\label{pr_explicit-formula-size-Id}
Let $n= p_1^{k_1}\cdots p_t^{k_t}$ be the canonical prime factorization of $n$, and $d\in SP(n)$. Then 
\[
|I_d(\mathbb{Z}_n)| =
\begin{cases}
	(p_i^{k_i} - p_i)\displaystyle\prod_{\substack{r=1 \\ r\neq i}}^{t} p_r^{k_r}, & \text{if $d = p_i^2$}, \\[0.5em]
	(p_i^{k_i} - 1)(p_j^{k_j} - 1)\displaystyle\prod_{\substack{r=1 \\ r\neq i,j}}^{t} p_r^{k_r}, & \text{if $d = p_i p_j$ with $i \neq j$}.
\end{cases}\]
\end{proposition}

\begin{proof}[\textbf{Proof}]
By Lemma~\ref{le_card-of-I_d}, we have $|I_d(\mathbb{Z}_n)| = \sum_{\substack{d\mid \ell\\ \ell\mid n}}\varphi(\ell)$. Let $\ell\in \N$ such that $d\mid \ell$ and $\ell\mid n$. Then $\ell$ can be uniquely expressed as $\ell = d_1 d_2$, where $d_{1}\mid p_{i}^{k_i} p_{j}^{k_j}$, $d_2 \mid \prod_{\substack{r=1\\r\neq i,j}}^{t}p_r^{k_r}$, and $(d_1, d_2) = 1$ (cf.~\cite[p.~108, Lemma]{burton-b07}). Furthermore, we have $\varphi(\ell) = \varphi(d_1 d_2) = \varphi(d_1)\varphi(d_2)$, since $\varphi$ is multiplicative. Since $d = p_ip_j$ for some $i,j\in \{1,\ldots, t\}$, we consider separately the two cases: 

\vspace{0.5mm}
\noindent\textbf{Case 1:} $i=j$. Then $d = p_{i}^{2}$. Since $d\mid d_1d_2$, the exponent of $p_i$ in $d_1$ must be at least $2$. It follows that $d_1 = p_i^{s}$, where $2 \leq s \leq k_{i}$. Therefore $\sum_{d_1 \mid p_{i}^{k_i}}\varphi(d_1) = \sum_{s=2}^{k_i}\varphi(p_i^{s}) = p_i^{k_i} - p_i$. Recall that $\varphi(\ell) = \varphi(d_1)\varphi(d_2)$. Therefore, 
\[\sum_{\ell\mid n}\varphi(\ell)= \sum_{\substack{d_{1}\mid p_{i}^{k_i} \\ d_2 \mid \prod_{\substack{r=1\\r\neq i}}^{t}p_r^{k_r}}} \varphi(d_1) \varphi(d_2) =\biggl(\sum_{d_1 \mid p_{i}^{k_i}}\varphi(d_1)\biggr) \biggl(\sum_{d_2\mid \prod_{\substack{r=1\\r\neq i}}^{t}p_r^{k_r}}\varphi(d_2)\biggr).\]
By \cite[Theorem~7.6]{burton-b07}, we have $\sum_{d_2\mid \prod_{\substack{r=1\\r\neq i}}^{t}p_r^{k_r}}\varphi(d_2) = \prod_{\substack{r=1\\r\neq i}}^{t}p_r^{k_r}$. It follows that

\[\sum_{\ell\mid n}\varphi(\ell)=(p_i^{k_i}-p_i)\prod_{\substack{r=1\\r\neq i}}^{t}p_r^{k_r}.\]
Hence, by Lemma~\ref{le_card-of-I_d}, we obtain
\[|I_{d}(\Z_{n})| = \sum_{\substack{d \mid \ell \\ \ell \mid n}} \varphi(\ell)= (p_i^{k_i}-p_i)\prod_{\substack{r=1\\r\neq i}}^{t}p_r^{k_r}.\]  

\vspace{0.5mm}
\noindent\textbf{Case 2:} $i\neq j$. Then $d = p_{i} p_{j}$, where $p_i, p_j$ are distinct. Since $d\mid d_1d_2$, the exponent of both $p_i$ and $p_j$ in $d_1$ must be at least $1$. It follows that $d_1 = p_i^{s}  p_j^{m}$, where $1 \leq s \leq k_i$ and $1 \leq m \leq k_j$. Therefore, since $\varphi(p_i^{s}  p_j^{m}) = \varphi(p_i^{s})\varphi(p_j^{m})$, we obtain
\[\sum_{d_1 \mid p_{i}^{k_i} p_{j}^{k_j}}\varphi(d_1)  =
\sum_{\substack{p_i^{s} \mid p_i^{k_i} \\ p_j^{m} \mid p_j^{k_j}}} \varphi(p_i^{s}) \varphi(p_j^{m})=\biggl(\sum_{s=1}^{k_i}\varphi(p_i^s)\biggr) \biggl(\sum_{m=1}^{k_j}\varphi(p_j^m)\biggr)= (p_i^{k_i} - 1) (p_j^{k_j} - 1).\]

Recall that $\varphi(\ell) = \varphi(d_1)\varphi(d_2)$. Thus,
\[\sum_{\ell\mid n}\varphi(\ell)= \sum_{\substack{d_{1}\mid p_{i}^{k_i} p_{j}^{k_j} \\ d_2 \mid \prod_{\substack{r=1\\r\neq i, j}}^{t}p_r^{k_r}}} \varphi(d_1) \varphi(d_2)
=\left(\sum_{d_1 \mid p_{i}^{k_i} p_{j}^{k_j}}\varphi(d_1)\right) \Biggl(\sum_{d_2\mid \prod_{\substack{r=1\\r\neq i,j}}^{t}p_r^{k_r}}\varphi(d_2)\Biggr).\]
 By \cite[Theorem~7.6]{burton-b07}, we have

 \[\sum_{d_2\mid \prod_{\substack{r=1\\r\neq i,j}}^{t}p_r^{k_r}}\varphi(d_2) = \prod_{\substack{r=1\\r\neq i,j}}^{t}p_r^{k_r}.\] It follows that 
\[\sum_{\ell\mid n}\varphi(\ell)= (p_i^{k_i} - 1) (p_j^{k_j} - 1)\prod_{\substack{r=1\\r\neq i,j}}^{t}p_r^{k_r}.\]

Hence, by Lemma~\ref{le_card-of-I_d}, we obtain
\[|I_{d}(\Z_{n})| = \sum_{\substack{d \mid \ell \\ \ell \mid n}} \varphi(\ell)=
(p_i^{k_i} - 1) (p_j^{k_j} - 1)\prod_{\substack{r=1\\r\neq i,j}}^{t}p_r^{k_r}.\]
\end{proof}


\begin{theorem}\label{th_ind-num-cyclic-group}
Let $n= p_1^{k_1}\cdots p_t^{k_t}$ be the canonical prime factorization of $n$. Then
\[
	\alpha\bigl(\Theta(\mathbb{Z}_n)\bigr) \;\geq\; \max\Biggl\{
	\,(p_i^{k_i} - p_i)   \prod_{\substack{r=1\\ r\neq i}}^t p_r^{k_r},\;
	(p_i^{k_i} - 1)(p_j^{k_j} - 1)\prod_{\substack{r=1\\r\neq i,j}}^t p_r^{k_r}
	\Biggr\},\]
	where the maximum is taken over all indices $i$ with $k_i \geq 2$ and all pairs $i \neq j$.

\vspace{0.2mm}
Moreover, if $k_1 = \cdots = k_t = 1$, then \[\alpha\!\bigl(\Theta(\mathbb{Z}_{n})\bigr)
\geq (p_{t-1}-1)(p_t-1) \prod_{r=1}^{t-2} p_r.\]
\end{theorem}

\begin{proof}[\textbf{Proof}]
The first inequality follows immediately from  Propositions~\ref{pr_lower-bound-for-ind-num} and \ref{pr_explicit-formula-size-Id}. Next, assume $k_1 = \cdots = k_t = 1$. Then  \[
\alpha\bigl(\Theta(\mathbb{Z}_n)\bigr) \;\geq\; \max_{1\leq i < j\leq t}\biggl\{
\,(p_i - 1)(p_j - 1)\prod_{\substack{r=1\\ r\neq i,j}}^t p_r
\biggr\} = n\cdot \max_{1\leq i < j\leq t}\left\{\left(1- \frac{1}{p_i}\right) \left(1- \frac{1}{p_j}\right)\right\}.\]

Since $p_1 < p_2 < \cdots < p_t$, we have $1-\frac{1}{p_1} <  1-\frac{1}{p_2}< \cdots <1-\frac{1}{p_t}$. Therefore,  
\[\max_{1\leq i< j\leq t}\left\{ \left(1- \frac{1}{p_i}\right) \left(1- \frac{1}{p_j}\right)\right\} =  \left(1- \frac{1}{p_{t-1}}\right) \left(1- \frac{1}{p_t}\right).\] Hence, 
\[\alpha\bigl(\Theta(\mathbb{Z}_n)\bigr) \geq  n \left(1- \frac{1}{p_{t-1}}\right) \left(1- \frac{1}{p_t}\right)
= (p_{t-1}-1)(p_t-1) \prod_{r=1}^{t-2} p_r.\]
\end{proof}


The following two examples illustrate that the lower bound in Theorem~\ref{th_ind-num-cyclic-group} need not be tight in general.

\begin{example}
Consider the cyclic group $\Z_{900}$ of order $900=2^2\cdot 3^2 \cdot 5^2$. By Theorem~\ref{th_ind-num-cyclic-group}, we obtain  $\alpha\bigl(\Theta(\mathbb{Z}_n)\bigr)\geq 768$. Now consider 
\[I= \bigl\{g\in \Z_{900} \colon |g|\in\{30, 50, 60, 75, 90, 100, 125,  150, 180, 300, 450, 900\}\bigr\}.\] It is clear that $I$ is independent in $\Theta(\Z_{900})$, since $(|x|,|y|)$ is composite for any distinct $x, y \in I$. By applying Theorem~\ref{th_phi-d-elems-in-Zn} together with Theorem~\ref{th_value-of-phi-n}, we deduce $|I|=796$. Therefore,
$\alpha\bigl(\Theta(\Z_{900})\bigr)\geq 796 > 768$.
\end{example}


\begin{example}
	Consider the cyclic group $\Z_{1155}$ of order $1155= 3 \cdot 5\cdot 7\cdot 11$. By Theorem~\ref{th_ind-num-cyclic-group}, we obtain $\alpha\bigl(\Theta(\Z_{1155})\bigr) \geq 900$. Now consider 
	\[I = \bigl\{g\in \Z_{1155}\colon |g|\in \{105, 165, 231, 385, 1155\}\bigr\}.\] Clearly $I$ is an independent set of $\Theta(\Z_{1155})$, since $(|x|, |y|)$ is composite for any distinct $x, y\in I$. By applying Theorem~\ref{th_phi-d-elems-in-Zn} together with Theorem~\ref{th_value-of-phi-n}, we deduce $|I|=968$. Therefore,
	$\alpha\bigl(\Theta(\Z_{1155})\bigr)\geq 968 > 900$.
	\end{example}

In the sequel, we establish five results demonstrating the tightness of the lower bound for $\alpha\bigl(\Theta(\mathbb{Z}_n)\bigr)$, as stated in Theorem~\ref{th_ind-num-cyclic-group}, whenever $n$ has at most two prime divisors.

\begin{corollary}
Let $n = p^m$, where $p$ is a prime and $m\geq 2$. Then $\alpha\bigl(\Theta(\Z_n)\bigr) = n-p$. 
\end{corollary}

\begin{proof}[\textbf{Proof}]
By Theorem~\ref{th_phi-d-elems-in-Zn}, $|P(\Z_n)| = 1 + \varphi(p) = p$. Hence, by Corollary~\ref{cor_ind-no-of-split-graph}, we obtain $\alpha\bigl(\Theta(\Z_n)\bigr) =  n-p$.
\end{proof}


\begin{corollary}
Let $n = pq$, where $p$ and $q$ are distinct primes. Then $\alpha\bigl(\Theta(\Z_n)\bigr) = (p-1)(q-1)$. 
\end{corollary}

\begin{proof}[\textbf{Proof}]
By Theorem~\ref{th_phi-d-elems-in-Zn}, $|P(\Z_n)| = 1 + \varphi(p) +\varphi(q) = p+q-1$. Hence, by Corollary~\ref{cor_ind-no-of-split-graph}, we obtain $\alpha\bigl(\Theta(\Z_n)\bigr) =  (p-1)(q-1)$.
\end{proof}

\begin{proposition}
Let $n= p_1p_2p_3$, where $p_1<p_2<p_3$ are primes. Then
	$\alpha\bigl(\Theta(\mathbb{Z}_{n})\bigr)= p_1 (p_2-1)(p_3-1)$.
\end{proposition}

\begin{proof}[\textbf{Proof}]
By \cite[Proposition~3.5]{shubh-arxiv25}, we have $\Theta(\mathbb{Z}_{n})\cong H(\mathcal{G})$, where 
\[\mathcal{G} = \bigl\{K_{p_1+p_2+p_3-2}, E_{(p_1-1)(p_2-1)}, E_{(p_2-1)(p_3-1)}, E_{(p_1-1)(p_3-1)}, E_{(p_1-1)(p_2-1)(p_3-1)}\bigr\}\] and the indexing graph $H$, depicted in Figure~3 of \cite{shubh-arxiv25}, is reproduced in Figure~\ref{graph-H-of-cyclic-pqr}.  

	\begin{figure}[h]
	\begin{tikzpicture}
		
		\draw[fill=black] (0,0) circle (1.4pt);
		\draw[fill=black] (0.75,0) circle (1.4pt);
		\draw[fill=black] (-0.75,0.75) circle (1.4pt);
		\draw[fill=black] (-0.75,-0.75) circle (1.4pt);
		\draw[fill=black] (1.5,0) circle (1.4pt);
		
		\draw[line width=0.6 pt] (0,0)--(0.75,0)--(-0.75,0.75)--(-0.75,-0.75)--(0.75,0);
		\draw[line width=0.6 pt] (-0.75,0.75)--(0,0)--(-0.75,-0.75);
		\draw[line width=0.6 pt] (0.75,0)--(1.5,0);

		\node at (0.85,0.2)  {$1$};
		\node at (-0.25,0) {$2$};
		\node at (-0.9,0.75) {$3$};
		\node at (-0.9,-0.75) {$4$};
		\node at (1.5,0.2) {$5$};
	\end{tikzpicture}
	\vspace{-3mm}
	\caption{Graph $H$}	\label{graph-H-of-cyclic-pqr}
\end{figure}
Observe that $\{1,2,3,4\}$ forms a maximum clique in $H$, so any independent set of $H$ can contain at most one vertex from this clique. Moreover, vertex $1$ is dominating in $H$. Consequently, the maximal independent sets of cardinality greater than one in $H$ are  precisely $\{2,5\}$, $\{3,5\}$, and $\{4,5\}$. It follows that the maximal independent sets of cardinality greater than one in  $H(\mathcal{G})$ are precisely $V\bigl(E_{(p_1-1)(p_2-1)} \sqcup E_{(p_1-1)(p_2-1)(p_3-1)}\bigr)$,
$V\bigl(E_{(p_2-1)(p_3-1)}\sqcup E_{(p_1-1)(p_2-1)(p_3-1)}\bigr)$, and $V\bigl(E_{(p_1-1)(p_3-1)}\sqcup E_{(p_1-1)(p_2-1)(p_3-1)}\bigr)$ with respective cardinalities $p_3(p_1-1)(p_2-1)$, $p_1(p_2-1)(p_3-1)$, and $p_2(p_1-1)(p_3-1)$.
Therefore,
\begin{align*}
	\alpha\bigl(H(\mathcal{G})\bigr)
	&=\max\bigl\{p_1(p_2-1)(p_3-1),\; p_2(p_1-1)(p_3-1),\; p_3(p_1-1)(p_2-1)\bigr\}\\
	&= n\cdot\max_{1\leq i < j \leq 3}\biggl\{\left(1- \frac{1}{p_i}\right)\left(1- \frac{1}{p_j}\right) \biggr\}.
\end{align*}
Since $p_1 < p_2 < p_3$, we have $1-\frac{1}{p_1} <  1-\frac{1}{p_2} <1-\frac{1}{p_3}$. It follows that
\[\max_{1\leq i < j \leq 3}\biggl\{\left(1- \frac{1}{p_i}\right)\left(1- \frac{1}{p_j}\right)\biggr\} =  \left(1- \frac{1}{p_{2}}\right) \left(1- \frac{1}{p_3}\right).\]
Thus, $\alpha\bigl(H(\mathcal{G})\bigr) = n \left(1- \frac{1}{p_{2}}\right) \left(1- \frac{1}{p_3}\right)
= p_1(p_2-1)(p_3-1)$. Hence, since $\Theta(\mathbb{Z}_{n})\cong H(\mathcal{G})$, we conclude that $\alpha\bigl(\Theta(\mathbb{Z}_{n})\bigr)= p_1(p_2-1)(p_3-1)$.
\end{proof}

\begin{proposition}\label{pr_ind-no-cyclic-grup-order-pqb}
Let $n=pq^{b}$, where $p$ and $q$ are distinct primes and $b\geq 2$. Then \[\alpha\bigl(\Theta(Z_n)\bigr)= n- \min\bigl\{pq,\; q^{b}+p-1\bigr\}.\]
\end{proposition}

\begin{proof}[\textbf{Proof}]
By \cite[Proposition~3.3]{shubh-arxiv25}, we have $\Theta(\Z_n)\cong H(\mathcal{G})$, where \[\mathcal{G} = \bigl\{ K_{p+q-1},\; E_{q^{b}-q},\; E_{(p-1)(q-1)},\; E_{(p-1)(q^{b}-q)} \bigr\}\]
and the indexing graph $H$, depicted in Figure~1 of \cite{shubh-arxiv25}, is reproduced in Figure~\ref{graph-H-of-cyclic-pqm}.  

\begin{figure}[h]
	\begin{tikzpicture}
		\draw[fill=black] (0,0) circle (1.4pt);
		\draw[fill=black] (-0.7,0) circle (1.4pt);
		\draw[fill=black] (0.8,0.4) circle (1.4pt);
		\draw[fill=black] (0.8,-0.4) circle (1.4pt);
		
		\draw[line width=0.6 pt] (-0.7,0) -- (0,0) -- (0.8,0.4)--(0.8, -0.4)--(0,0);
		
		\node at (0,0.2) {$1$};
		\node at (0.95, 0.4) {$2$};
		\node at (0.95, -0.4) {$3$};
		\node at (-0.7, 0.2) {$4$};
	\end{tikzpicture}	
	\caption{Graph $H$}	\label{graph-H-of-cyclic-pqm}
\end{figure}

Observe that $\{1,2,3\}$ forms a maximum clique in $H$, so any independent set of $H$ can contain at most one vertex from this clique. Moreover, vertex $1$ is dominating in $H$. Consequently, the maximal independent sets of cardinality greater than one in $H$ are  precisely $\{2,4\}$ and $\{3,4\}$. It follows that the maximal independent sets of cardinality greater than one in $H(\mathcal{G})$ are precisely
$V\bigl(E_{q^{b}-q} \sqcup E_{(p-1)(q^{b}-q)}\bigr)$ and $V\bigl(E_{(p-1)(q-1)} \sqcup E_{(p-1)(q^{b}-q)}\bigr)$ with respective cardinalities $n-pq$ and $n-(q^{b}+p-1)$.
Therefore,
\[\alpha\bigl(H(\mathcal{G})\bigr)
=\max\bigl\{n-pq,\; n-(q^{b}+p-1)\bigr\} = n- \min\bigl\{pq,\; q^{b}+p-1\bigr\}.\]
Hence, since $\Theta(\mathbb{Z}_{n})\cong H(\mathcal{G})$, we conclude
$\alpha\bigl(\Theta(\Z_n)\bigr) = n- \min\bigl\{pq,\; q^{b}+p-1\bigr\}$.
\end{proof}

\begin{proposition}
Let $n = p^{a}q^{b}$, where $p$ and $q$ are distinct primes and $a, b \geq 2$. Then
\[\alpha\bigl(\Theta(\Z_n)\bigr)=n-\min\bigl\{p^{a}q,\ pq^{b},\ p^{a}+q^{b}-1\bigr\}.\]
\end{proposition}

\begin{proof}[\textbf{Proof}]
By \cite[Proposition~3.4]{shubh-arxiv25}, we have $\Theta(\Z_n)\cong H(\mathcal{G})$, where 
\[\mathcal{G} = \bigl\{ K_{p+q-1}, E_{p^{a}-p}, E_{q^{b}-q}, E_{(p-1)(q-1)}, E_{(p^{a}-p)(q-1)}, E_{(p-1)(q^{b}-q)}, E_{(p^{a}-p)(q^{b}-q)} \bigr\}\]
and the indexing graph $H$, depicted in Figure~2 of \cite{shubh-arxiv25}, is reproduced in Figure~\ref{graph-H-of-cyclic-pmqn}. 

\begin{figure}[h!]
	\begin{tikzpicture}[very thick]
		
		\draw[fill=black] (0,0) circle (1.4pt);
		\draw[fill=black] (2,-0.75) circle (1.4pt);
		\draw[fill=black] (2,0.75) circle (1.4pt);
		\draw[fill=black] (1,0) circle (1.4pt);
		\draw[fill=black] (-1,0.75) circle (1.4pt);
		\draw[fill=black] (-1,-0.75) circle (1.4pt);
		\draw[fill=black] (-1,0) circle (1.4pt);
		\draw[line width=0.6 pt] (0,0)--(-1,-0.75)--(2,-0.75)--(2,0.75)--(-1,0.75)--(0,0);
		\draw[line width=0.6 pt] (-1,0)--(0,0)--(1,0);
		\draw[line width=0.6 pt] (0,0)--(2,-0.75)--(1,0)--(2,0.75)--(0,0);

		\node at (0,0.25)  {$1$};
		\node at (2.2,-0.75) {$2$};
		\node at (2.2,0.75) {$3$};
		\node at (1.3,0) {$4$};
		\node at (-1.2,0.75)  {$5$};
		\node at (-1.2,-0.75) {$6$};
		\node at (-1.2,0) {$7$};
	\end{tikzpicture}
	\vspace{-3mm}
	\caption{Graph $H$}	\label{graph-H-of-cyclic-pmqn}
\end{figure}

Observe that $\{1,2,3,4\}$ forms a maximum clique in $H$, so any independent set of $H$ can contain at most one vertex from this clique. Moreover, vertex $1$ is dominating in $H$. Consequently, the maximal independent sets of cardinality greater than one in $H$ are precisely $\{2,5,7\}$, $\{3,6,7\}$, and $\{4, 5, 6, 7\}$. It follows that the maximal independent sets of cardinality greater than one in $H(\mathcal{G})$ are precisely $V\bigl(E_{p^{a}-p} \sqcup E_{(p^{a}-p)(q-1)}\sqcup E_{(p^{a}-p)(q^{b}-q)}\bigr)$, $V\bigl(E_{q^{b}-q} \sqcup E_{(q^{b}-q)(p-1)}\sqcup E_{(p^{a}-p)(q^{b}-q)}\bigr)$, and $V\bigl(E_{(p-1)(q-1)}\sqcup E_{(p^{a}-p)(q-1)}\sqcup E_{(p-1)(q^{b}-q)}\sqcup E_{(p^{a}-p)(q^{b}-q)}\bigr)$ with respective cardinalities $n - pq^{b}$, $n- p^{a}q$, and $n - (p^{a}+q^{b}-1)$. 
Therefore,
\[\alpha\bigl(H(\mathcal{G})\bigr)
=\max\bigl\{n - pq^{b},\ n- p^{a}q,\ n - (p^{a}+q^{b}-1)\bigr\} = n-\min\bigl\{p^{a}q,\; pq^{b},\; p^{a}+q^{b}-1\bigr\}.\]
Hence, since $\Theta(\mathbb{Z}_{n})\cong H(\mathcal{G})$, we obtain
$\alpha\bigl(\Theta(\Z_n)\bigr)= n-\min\bigl\{p^{a}q,\ pq^{b},\ p^{a}+q^{b}-1\bigr\}$.
\end{proof}


We now turn to determining the independence number of $\Theta(D_{2n})$, where $D_{2n}$ denotes the dihedral group of order $2n$. For each integer $n\geq 3$, recall that the dihedral group $D_{2n}$ is a finite nonabelian group of order $2n$ defined by the presentation \[D_{2n} = \bigl\langle a, b \ \colon \ a^n =1,\  b^2=1, \ bab^{-1}= a^{-1}\bigr\rangle,\] where $1$ denotes the identity element of $D_{2n}$ (cf.~\cite[p.~181]{james-b01}). 

\vspace{0.5mm}

The following proposition establishes the independence number of $\Theta(D_{2n})$ for any integer $n\geq 3$.

\begin{proposition}
For every integer $n\geq 3$,  $\alpha\bigl(\Theta(D_{2n})\bigr)=\alpha\bigl(\Theta(\Z_n)\bigr)$. 
\end{proposition}

\begin{proof}[\textbf{Proof}]
By \cite[Lemma~5.2]{shubh-arxiv25}, we have $\Theta(D_{2n})\cong\Theta(\mathbb{Z}_n)\vee K_n$. Note that $\alpha(K_n)= 1$. Hence,
$\alpha\bigl(\Theta(D_{2n})\bigr)= \alpha\bigl(\Theta(\Z_n)\vee K_n\bigr) = \max\left\{\alpha\bigl(\Theta(\Z_n)\bigr),\ \alpha(K_n)\right\} = \alpha\bigl(\Theta(\Z_n)\bigr)$.
\end{proof}


We next turn to determining the independence number of $\Theta(Q_{4n})$, where $Q_{4n}$ denotes the dicyclic group of order $4n$. For each integer $n\geq 3$, recall that the dicyclic group $Q_{4n}$ is a nonabelian group of order $4n$ defined by the presentation \[Q_{4n} = \bigl\langle a, b \ \colon \ a^{2n} = 1,\ b^2 = a^n,\ bab^{-1} = a^{-1}\rangle,\] where $1$ denotes the identity element of $Q_{n}$ (cf.~\cite[p~178]{james-b01}). The elements of $Q_{4n}$ can be partitioned into two disjoint subsets of equal size as follows:
\[\langle a \rangle = \bigl\{1,a,a^2, \ldots, a^{2n-1}\bigr\} \quad \text{ and } \quad Q_{4n}\setminus \langle a \rangle = \bigl\{b,ab,a^2b, \ldots, a^{2n-1}b\bigr\}.\]
Moreover, $|a^ib| = 4$ for all $i\in \{0,1,\ldots, 2n-1\}$.

\vspace{0.5mm}

The following proposition establishes the independence number of $\Theta(Q_{4n})$ for all odd integers $n\geq 3$.

\begin{proposition}
For every odd integer $n\geq 3$, $\alpha\bigl(\Theta(Q_{4n})\bigr)=2n$. 
\end{proposition}
\begin{proof}[\textbf{Proof}]
Assume $n\geq 3$ is odd. Let $H = \langle a \rangle$ denote the cyclic subgroup of order $2n$ in $Q_{4n}$. Clearly $H \cong \Z_{2n}$, and so $\Theta(Q_{4n})[H]\cong \Theta(\Z_{2n})$. Furthermore, observe that $\Theta(Q_{4n})[Hb]\cong E_{2n}$, since $|Hb| = 2n$ and every element of $Hb$ has order $4$. 

\vspace{0.4mm}
Now since $n$ is odd, the subgroup $H$ does not contain any element whose order is divisible by $4$. Therefore, $(4, |x|)\in \{1,2\}$ for all $x\in H$. Consequently, $\Theta(Q_{4n})[H]\sim \Theta(Q_{4n})[Hb]$, and so $\Theta(Q_{4n}) =  \Theta(Q_{4n})[H] \vee \Theta(Q_{4n})[Hb]$. Since $\Theta(Q_{4n})[H]\cong \Theta(\Z_{2n})$ and $\Theta(Q_{4n})[Hb]\cong E_{2n}$, it follows that $\Theta(Q_{4n}) \cong \Theta(\Z_{2n})\vee E_{2n}$. Hence, since $\alpha(E_{2n}) = 2n$, we obtain
\[\alpha\bigl(\Theta(Q_{4n})\bigr) = \alpha\bigl(\Theta(\Z_{2n})\vee E_{2n}\bigr) = \max\bigl\{\alpha\bigl(\Theta(\Z_{2n})\bigr),\, \alpha(E_{2n})\bigr\} = 2n.\]
\end{proof}


The following proposition establishes the independence number of $\Theta(Q_{4n})$ for all even integers $n\geq 3$.

\begin{proposition}
Let $n\geq 3$ be an even positive integer. Suppose that $n= 2^k m$, where $k\geq 1$ and $m$ is an odd positive integer. Then 
\[\alpha\bigl(\Theta(Q_{4n})\bigr)=4n-2m.\]
\end{proposition}

\begin{proof}[\textbf{Proof}]
Since $4\in SP(4n)$ and $I_4(Q_{4n})\neq \varnothing$,  by Lemma~\ref{le_Id-maximal-set} the set $I_4(Q_{4n})$ is a maximal independent set of $\Theta(Q_{4n})$. Let $H = \langle a \rangle $ denote the cyclic subgroup of order $2n$ in $Q_{4n}$. Since $n$ is even, by Proposition~\ref{pr_explicit-formula-size-Id}, we have $|I_4(Q_{4n})\cap H| = (2^{k+1}-2)m = 2n-2m$. Moreover, observe that $bH \subseteq I_4(Q_{4n})$, since every element of $bH$ has order $4$. Therefore, since $|bH| =2n$, we obtain
\[|I_4(Q_{4n})| = |I_4(Q_{4n})\cap H| + |bH|= (2n-2m)+2n= 4n-2m.\]

Finally, we show that $I_4(Q_{4n})$ is a maximum independent set of $\Theta(Q_{4n})$. Let $I$ be a maximum independent set of $\Theta(Q_{4n})$. Then $|I|\geq 4n-2m$. Since $|H|=2n$ and $4n-2m> 2n$, it follows that $I\cap bH \neq \varnothing$. Now let $x\in I$. Since $I$ is independent and the order of every element of $I\cap bH$ is $4$, we obtain $(4, |x|)= 4$. Thus $4$ divides $|x|$, which implies $x\in I_4(Q_{4n})$. Therefore $I\subseteq I_4(Q_{4n})$. Since $I$ is maximum in $\Theta(Q_{4n})$, it follows that $I=I_4(Q_{4n})$. Hence 
$\alpha\bigl(\Theta(Q_{4n})\bigr)= |I_4(Q_{4n})| = 4n-2m$.
\end{proof}

Finally, we determine the independence number of $\Theta(SD_{8n})$, where $SD_{8n}$ denotes the semidihedral group of order $8n$. For each integer $n\geq 3$, recall that the semidihedral group $SD_{8n}$ is a finite nonabelian group of order $8n$ defined by the presentation
\[SD_{8n}=\bigl\langle a,b \ \colon \ a^{4n}=1,\; b^2=1,\; bab^{-1}=a^{2n-1} \bigr\rangle,\] where $1$ denotes the identity element of $SD_{8n}$. The elements of $SD_{8n}$ can be partitioned into two disjoint subsets of equal size as follows:
\[\langle a \rangle = \bigl\{1,a,a^2, \ldots, a^{4n-1}\bigr\} \quad \text{ and } \quad SD_{8n}\setminus \langle a \rangle = \bigl\{b,ab,a^2b, \ldots, a^{4n-1}b\bigr\}.\]
Moreover, for every $i\in \{0,1,\ldots, 4n-1\}$,
\begin{equation*}
	|a^ib| =
	\begin{cases}
		2 & \text{if $i$ is even},\\
		4 & \text{if $i$ is odd}.
	\end{cases}
\end{equation*}

\begin{proposition}
Let $n\geq 3$ be an integer. Suppose that $n = 2^{k} m$, where $k \geq 0$ and $m$ is a positive odd integer. Then \[\alpha\bigl(\Theta(SD_{8n})\bigr)=6n-2m.\] 
\end{proposition}

\begin{proof}[\textbf{Proof}]
Since $4\in SP(8n)$ and $I_4(SD_{8n})\neq \varnothing$, by Lemma~\ref{le_Id-maximal-set} the set $I_4(SD_{8n})$ is a maximal independent set of $\Theta(SD_{8n})$. Let $H=\langle a \rangle$ denote the cyclic subgroup of order $4n$ in $SD_{8n}$. By Proposition~\ref{pr_explicit-formula-size-Id}, we obtain $|I_4(SD_{8n})\cap H|=4n-2m$. Moreover, the coset $Hb\;  \bigl(= SD_{8n}\setminus \langle a \rangle\bigr)$ contains $4n$ elements, of which $2n$ are of order $2$ and $2n$ are of order $4$. It follows that $|I_4(SD_{8n})\cap Hb|=2n$. Therefore
\[|I_4(SD_{8n})| = |I_4(SD_{8n})\cap H|+ |I_4(SD_{8n})\cap Hb| = (4n-2m) + 2n= 6n-2m.\]

Finally, we show that $I_4(SD_{8n})$ is a maximum independent set of $\Theta(SD_{8n})$. Let $I$ be a maximum independent set of $\Theta(SD_{8n})$. Then $|I| \geq 6n-2m$. Since $6n-2m \geq 4n$, it follows that $|I|\geq 4n$. We claim that $I\cap Hb\neq\varnothing$. Suppose to the contrary that $I\cap Hb=\varnothing$. Then $I \subseteq H\setminus P(H)$. Since $|H| = 4n$ and $|P(H)| \geq 2$, it follows that $|I| < 4n$, a contradiction. Hence $I\cap Hb\neq\varnothing$. Note that $Hb$ contains elements of order either $2$ or $4$. Thus, since every element of order $2$ is dominating in $\Theta(SD_{8n})$, there exists $g\in Hb\cap I$ with $|g| = 4$. Since $I$ is independent, we have $(4, |x|)=4$ for all $x\in I$. Consequently, $4$ divides $|x|$, which implies $x\in I_4(SD_{8n})$. Therefore $I \subseteq I_4(SD_{8n})$. Since $I$ is maximum, it follows that $I = I_4(SD_{8n})$. Hence, $\alpha\bigl(\Theta(SD_{8n})\bigr)= |I_4(SD_{8n})| =6n-2m$.  
\end{proof}




\end{document}